\documentclass[10pt, reqno]{amsart}
\usepackage{amssymb,amsmath,amsthm,color, cite}
\theoremstyle{plain}
\newtheorem{theorem}{Theorem}[section]

\theoremstyle{remark}

\usepackage
[bookmarks=true,
   bookmarksnumbered=false,
   bookmarkstype=toc]
{hyperref}

\numberwithin{equation}{section}

\newcommand{\C}{\mathbb{C}}
\newcommand{\R}{\mathbb{R}}

\newcommand{\N}{\mathbb{N}}

\newcommand{\F}{\mathcal{F}}

\renewcommand{\Re}{\operatorname{Re}}
 
\renewcommand{\S}{\mathcal{S}}

\newcommand{\eps}{\varepsilon}
\newcommand{\fy}{\varphi}

\DeclareMathOperator{\supp}{supp}

\newcommand{\qtq}[1]{\quad\text{#1}\quad}

\newcommand{\pt}{\partial}
\newcommand{\ti}{\tilde}
\newcommand{\LR}[1]{{\langle #1 \rangle}}
\newcommand{\M}{\mathcal{M}}

\begin{document}

\title[Failure of scattering]{Failure of scattering to solitary waves for long-range nonlinear Schr\"odinger equations}

\author{Jason Murphy}
\address{Department of Mathematics and Statistics, Missouri University of Science \& Technology, Rolla, MO, USA}
\email{jason.murphy@mst.edu}

\author{Kenji Nakanishi}
\address{Research Institute for Mathematical Sciences,
Kyoto University}
\email{kenji@kurims.kyoto-u.ac.jp}

\begin{abstract} We consider nonlinear Schr\"odinger equations with either power-type or Hartree nonlinearity in the presence of an external potential.  We show that for long-range nonlinearities, solutions cannot exhibit scattering to solitary waves or more general localized waves.  This extends the well-known results concerning non-existence of non-trivial scattering states for long-range nonlinearities.
\end{abstract}

\maketitle

\section{Introduction}

We consider nonlinear Schr\"odinger equations of the form
\begin{equation}\label{nls}
 i\pt_t u  + \Delta u = Vu + F(u)
\end{equation}
for a complex-valued function of space-time $u:\R_t\times\R_x^d\to\C$ with $d\geq 1$. 

We take $F(u)$ to be either a power-type or Hartree-type nonlinearity, i.e.
\begin{equation}\label{F1}
F(u) = \mu |u|^p u \qtq{or} F(u) = \mu (|x|^{-\frac{dp}{2}}\ast |u|^2)u.
\end{equation}
We are interested in `long-range' nonlinearities, corresponding to $0<p\leq \tfrac{2}{d}$.  

We take $V:\R_t\times\R_x^d\to\C$ to be an external potential, which may be complex-valued and time-dependent.  For a given choice of $p$, we impose the following constraints on the potential:
\begin{equation} \label{cond V}
V\in L^\infty(\R;X),\qtq{where}
  X=L^{\frac{2}{p}-}(\R^d)+\begin{cases}
L^{\frac{d}{2}}(\R^d) &(d\ge 3),\\
L^{1+}(\R^2) &(d=2),\\
\M(\R)&(d=1).
 \end{cases} 
\end{equation}
Here and below we write $a\pm$ to denote $a\pm\eps$ for some sufficiently small $\eps>0$.  We exclude the $L_t^\infty L_x^1$ endpoint when $d=2$ due to the failure of the endpoint Strichartz estimate in that dimension. We write $\M(\R^d)$ for the Banach space of complex Radon measures with finite variation on $\R^d$.  By considering the extension to measures in the case $d=1$, we can include some interesting cases such as the delta potential (or even multiple delta potentials), as we will discuss in more detail below.  Under the assumptions of our main result (see Theorem~\ref{T2} below), we will always have $\max\{1,\tfrac{d}{2}\}<\tfrac{2}{p}$.

In the absence of an external potential (i.e. when $V\equiv 0$), it is well-known that for long-range nonlinearities, no solution (other than zero) can exhibit asymptotically linear behavior \cite{Strauss, Barab, Glassey}.  Related results have since been established in many different settings; see for example \cite{BSV, CO, MM, HLN, HNN, Shimomura, ST}.  The purpose of this note is to establish an analogous result concerning scattering to solitary waves or more general waves that are localized in space.  

Our main result is the following.

\begin{theorem}\label{T2} Let $\mu\in\C\backslash\{0\}$, $0<p<1$, and $p\le \tfrac{2}{d}$. 
Suppose that $u(t)$ is a forward-global solution to \eqref{nls} with \eqref{F1} and \eqref{cond V} that admits a decomposition of the form
\begin{equation}\label{u-decomposition}
 u(t) = l(t) + e^{it\Delta}v_+  + o(1)\qtq{in}L^2(\R^d) \qtq{as}t\to\infty, 
\end{equation}
where $l\in L^\infty_t((0,\infty);(L^2\cap L^q)(\R^d))$ for some $q<2$ and $v_+\in L^2(\R^d)$.  Then $v_+\equiv 0$. 
\end{theorem}
For the more precise meaning of solution, as well as the local well-posedness for \eqref{nls} in $L^2(\R^d)$, see Section \ref{S:preliminaries}. 

The function $l(t)$ in Theorem~\ref{T2} may take on a wide range of forms.  We have in mind a superposition of waves that are localized for all time: solitary waves, breathers, quasi-periodic solutions, almost periodic solutions, and so on (provided the equation \eqref{nls} supports such solutions!).

The analogue of Theorem~\ref{T2} with $V=l=0$ was first established for power-type nonlinearities with $p\leq \min\{1,\tfrac{2}{d}\}$ and for the Hartree nonlinearity with $d=3$ and $p=\tfrac{2}{d}$ \cite{Strauss, Glassey}. Subsequently, \cite{Barab} filled in the gap $d=1$ and $1<p\leq 2$ for the power-type case.  All of these works were based on an argument of Glassey \cite{GlasseyKG}, which is also at the heart of our arguments below.  While we found no reference that explicitly treats the remaining Hartree nonlinearities, we remark that the same argument can be modified to handle the nonlinearities $F(u)=\mu(|x|^{-\frac{dp}{2}}\ast |u|^2)u$ for $0<p\leq1$ and $p\leq \tfrac{2}{d}$.

The works \cite{Strauss, Glassey, Barab} require assumptions  beyond merely $u,v_+\in L^2$.  In particular, these works assume that $v_+\in L^1$ in order to access the dispersive estimate, while \cite{Barab} additionally requires $u(0)\in H^1\cap \F H^1$ (and a defocusing nonlinearity) in order to utilize the pseudoconformal energy estimate.  In \cite[Theorem~7.5.2]{Cazenave}, the result is stated for power-type nonlinearities with $p\leq \min\{1,\tfrac{2}{d}\}$ and $u_0\in H^1\cap \F H^1$, but an examination of the proof shows that one needs only $u,v_+\in L^2$.  On the other hand, the weighted assumption seems to be necessary when addressing the remaining cases $d=1$ and $1<p\leq 2$.

In Theorem~\ref{T2}, our scattering assumption is solely in $L^2$, and no further decay is assumed on $u$ or the scattering state $v_+$.  Consequently our arguments break down when $p>1$ (as in previous works).  In fact, our arguments already break down at $p=1$ when treating the soliton part as well as the potential (cf. the proof of Theorem~\ref{T2} below).  By imposing some additional assumptions on $l(t)$ and $V(t)$, however, we can recover  a similar result for power-type case with $p=1\leq\tfrac{2}{d}$.  Moving beyond this range without imposing additional assumptions on $v_+$ or $u$ seems to be a difficult problem.  The precise result we will prove is the following.

\begin{theorem}\label{T3} Let $\mu\in\C\backslash\{0\}$ and $p=1\leq\tfrac{2}{d}$.  Suppose $u(t)$ is a forward-global solution to \eqref{nls} with a power-type nonlinearity in \eqref{F1} and $V\in L_t^\infty(L^2\cap L^{2-})$.  Suppose that $u$ admits a decomposition of the form \eqref{u-decomposition} with $v_+\in L^2$ and $l\in L_t^\infty(L^2\cap L^{2-})$. Assume additionally that there exists a measure $\nu\in\M(\R^d)$ that is singular with respect to Lebesgue measure and obeys
\begin{equation}\label{T3-hypothesis}
\langle\nu,\varphi\rangle \geq  \limsup_{t\to\infty} \langle t^d\bigl[|l(t,tx)|^2+|V(t,tx)|^2\bigr],\varphi\rangle \end{equation}
for any nonnegative $\varphi\in C_0(\R^d)$\footnote{We write $C_0$ for the closed subspace of $L^\infty$ consisting of continuous functions decaying at $|x|\to\infty$.}. Then $v_+\equiv 0$.
\end{theorem}

While the assumptions in Theorem~\ref{T3} are stronger than the ones in Theorem~\ref{T2}, they do in fact hold in the types of situations that we are interested in.  For example, the assumptions on $l$ and $V$ hold if we take $l$, as well as $V$, in the form
\[
l(t) = \sum_{k=1}^N l_k(t,x-c_k(t)) \qtq{with} \sup_{k,\delta>0}\lim_{t\to\infty} \|l_k(t,x)\|_{L^2(|x|>\delta t)} = 0,
\]
where each $l_k(t)$ is bounded in $L^2(\R^d)\cap L^{2-}(\R^d)$, and each $c_k(t)\in\R^d$ has the `limit velocity' $v_k:=\lim_{t\to\infty}c_k(t)/t\in\R^d$. In this case, the optimal $\nu$ is given by  
\[
 \nu=\sum_{k=1}^N \delta(x-v_k)\limsup_{t\to\infty}\|l_k(t)\|_{L^2(\R^d)}^2,
 \]
in addition to the corresponding components for $V$ if it is present.  This covers cases such as multi-solitons, breathers, and even more general localized waves, which are allowed to spread in $x$ sublinearly as $t\to\infty$.  On the other hand, if the number $N$ is increasing to $\infty$ or some of velocities $c_k(t)/t$ are oscillating as $t\to\infty$, then the assumption may be violated. 

As mentioned above, the generality of our assumptions on $V$ for $d=1$ permits us to include interesting cases such as the delta potential, which amounts to choosing $V$ to be a point mass at the origin.  In particular, we may compare our results to the recent results of \cite{CuccagnaMaeda}, in which the authors treat the NLS with a delta potential and power-type nonlinearities $|u|^p u$.  They show that for small $H^1$ initial data and any $p>0$, the solution admits a unique decomposition of the form
\[
u(t) = l(t) + v(t)
\]  
as $t\to\infty$.  In the case of an attractive potential, their $l(t)$ is a nonlinear bound state parametrized by some small $z(t)\in\C$. In treating the full range $p>0$, they select $v$ to be orthogonal to the linear bound state and are able to prove a global space-time estimate for $v$, specifically $e^{-\gamma |x|}v\in L_t^2 H_x^1$ for some $\gamma>0$.  For $p>1$, they use a refined orthogonality condition to define $v$.  Then, in addition to the space-time estimate, they can prove convergence of the parameter $z(t)$ up to a phase (`selection of the ground state').  For a repulsive potential, they have $l(t)=0$ and their result consists of the exponentially-weighted space-time estimate for $u$ itself.  Our Theorems~\ref{T2}~and~\ref{T3} assert that while the `dispersive' part $v(t)$ may decay, it will not contain a scattering component in the range $0<p\leq 1$.  The appearance of the exponent $p=1$ in both \cite{CuccagnaMaeda} and the present work appears to be coincidental.

Theorems~\ref{T2} and \ref{T3} provide an extension of the well-known results concerning non-existence of linear scattering for long-range nonlinearities to the more general setting of scattering to solitary waves, which is an ongoing area of active research interest.  As we will see, appropriate extensions of the arguments of \cite{Strauss, Barab, Glassey, GlasseyKG} suffice to show that the presence of long-range nonlinearities precludes the possibility of non-trivial scattering components in the description of the long-time behavior of solutions.

The rest of this paper is organized as follows: In Section~\ref{S:preliminaries}, we set up notation and collect some basic results.  In Section~\ref{S:T2}, we prove the main result, Theorem~\ref{T2}.  Finally, in Section~\ref{S:Extension}, we prove Theorem~\ref{T3}.  

\subsection*{Acknowledgements} J.~M. was supported by a Simons Collaboration grant. K.~N. was supported by JSPS KAKENHI Grant Number JP17H02854.

\section{Preliminaries}\label{S:preliminaries}

We use $C$ to denote a positive constant that may change from line to line, as well as the big-O notation.  We write $a\pm$ to denote $a\pm\eps$ for sufficiently small $\eps>0$.  We employ the standard little-o notation, with the following extension: if $X$ is a Banach space, then $f(t)=o(X)$ means $f(t)\to 0$ in $X$ as $t\to\infty$.

We denote the H\"older dual of $r\in[1,\infty]$ by $r'$.  The $L^2$ inner product is denoted by 
\[
 \LR{f,g} = \int_{\R^d} f(x)\overline{g(x)}dx,
 \]
as well as its extensions as duality pairing. 
 
We write $\S(\R^d)$ to denote the Schwartz class. We use the Lorentz spaces $L^{p,q}(\R^d)$ for $1<p<\infty$ and $1\le q\le\infty$, which can be defined as the real interpolation Banach space
\[ 
 L^{p,q} = (L^\infty, L^1)_{\frac1p,q} 
 \] 
 (see e.g. \cite{BL}).
The closed subspaces of $L^\infty(\R^d)$ consisting of bounded continuous functions and continuous functions decaying at $|x|\to\infty$ are denoted by 
\begin{align*}
 & C_b(\R^d):=\{f:\R^d\to \C\text{: continuous and bounded}\},
 \\& C_0(\R^d):=\{f:\R^d\to \C\text{: continuous, }\lim_{|x|\to\infty}f(x)=0\},
\end{align*}
respectively.

We define the Fourier transform by
\[
\F f(\xi)=\hat f(\xi)=(2\pi)^{-\frac{d}{2}}\int_{\R^d} e^{-ix\xi}f(x)\,dx. 
\]
We write $e^{it\Delta}$ for the Schr\"odinger group, i.e. $e^{it\Delta}=\F^{-1} e^{-it|\xi|^2}\F$.  One may also write $e^{it\Delta}$ as a convolution operator, namely
\[
[e^{it\Delta}f](x) = (4\pi i t)^{-\frac{d}{2}}\int_{\R^d} e^{i|x-y|^2/4t}f(y)\,dy.
\] 
Introducing the modulation and dilation operators
\[
M(t) = e^{i|x|^2/4t} \qtq{and} [D(t)f](x) = (2it)^{-\frac{d}{2}}f(\tfrac{x}{2t}),
\]
we can then read off the useful factorization
\begin{equation}\label{factorization}
e^{it\Delta} = MD\F M
\end{equation}
(cf. \cite[Remark 2.2.5]{Cazenave}), where we have allowed ourselves (here and below) to omit the explicit dependence on $t$.  

We next recall the dispersive estimate for the free Schr\"odinger equation: 
\[
 \|e^{it\Delta}f\|_{C_b(\R^d)} \leq C |t|^{-\frac{d}{2}}\|f\|_{\M(\R^d)},
 \]
which follows immediately from the explicit formula of $e^{it\Delta}$ by convolution in $x$. 
Note, however, that we cannot replace $C_b$ with $C_0$, cf. the counterexample $f=\delta$. 
By complex interpolation, the above estimate implies 
\[
\|e^{it\Delta}f\|_{L^r(\R^d)}\leq C|t|^{-(\frac{d}{2}-\frac{d}{r})}\|f\|_{L^{r'}(\R^d)},\quad r\in[2,\infty], 
\]
which is one of the key ingredients in this paper.  These dispersive estimates also imply the well-known Strichartz estimates for $e^{it\Delta}$ (cf.~\cite{GV,KT,Strichartz}). We note that when $d=1$ we obtain the endpoint estimate
\begin{equation} \label{Stz 1D}
\|u\|_{L^4_t C_b} \leq C\{\|u_0\|_2 +\|(i\pt_t+\Delta)u\|_{L^{4/3}_t \M}\}.
\end{equation}

Strichartz estimates are the key ingredient for the well-posedness theory of \eqref{nls}. In particular, local well-posedness of \eqref{nls} in $L^2(\R^d)$  follows for $0<p<\tfrac{4}{d}$ with \eqref{cond V} via the well-known argument using Strichartz estimates and the Banach fixed point theorem.  See for example \cite[Corollary~4.6.5]{Cazenave}.  In that work, the potential is a time-independent function, but the argument applies equally well to the time-dependent case, as well as the case of measure-valued potentials in $d=1$ (using \eqref{Stz 1D} in place of the usual $L^\infty$-$L^1$ Strichartz estimate). Indeed, the assumption \eqref{cond V} on $V$ is more than enough for $L^2$ well-posedness; in fact, one may weaken it by replacing the space $L^{\frac{2}{p}-}(\R^d)$ with $L^\infty(\R^d)$.  

The existence time of solutions is thereby bounded from below uniformly with respect to the initial $L^2$ norm; moreover, the $L^2$ norm controls all Strichartz norms of the solution over the interval of local existence.  More precisely, for any $B>0$ there exists $T>0$ such that for any $u_0\in L^2(\R^d)$ with $\|u_0\|_2\le B$, \eqref{nls} has a unique solution $u$ on $[0,T]$ with $u(0)=u_0$ satisfying 
\[
 C\|u_0\|_2 \geq \|u\|_{ST(0,T)}:=\|u\|_{L^\infty_t L^2_x(0,T)} 
 +\begin{cases} \|u\|_{L^2_tL_x^{\frac{2d}{d-2}}(0,T)} &(d\ge 3),\\ \|u\|_{L^q_tL_x^r(0,T)} &(d=2),\\ \|u\|_{L^4_t C_b(0,T)} &(d=1), \end{cases} \]
where for the $d=2$ case we can choose any $\tfrac{1}{q}+\tfrac{1}{r}=\tfrac{1}{2}$ with $q>2$, and $X(0,T)$ denotes the restriction onto the time interval $(0,T)$ of the Banach space $X$ of space-time functions. 

The unique local solutions in $L^2$ satisfy the equation \eqref{nls} in the distribution sense: For any $\fy\in\S(\R^d)$, we have 
\[
 i\pt_t\LR{u(t),\fy} = \LR{-\Delta u(t)+V(t)u(t)+F(u(t)),\fy},\]
where the right side is locally integrable in $t$, thanks to the Strichartz estimates. 
We note that regarding the nonlinear term as a source, this yields the same solution as defined by the  $C^0$ unitary group $e^{it(\Delta-V)}$ of the self-adjoint operator $-\Delta+V$ in the case of $V:\R^d\to\R$. 
In fact, for $V\in(L^{d/2}+L^\infty)(\R^d)$ in $d\ge 2$ and $V\in(\M+L^\infty)(\R)$ in $d=1$, 
the Schr\"odinger operator $-\Delta+V$ is defined by the KLMN theorem (cf.~\cite[Theorem X.17]{RS2} as well as \cite[Example X.2.3]{RS2} for the case $V=\delta$) as a perturbation of $-\Delta$ by the quadratic form $\LR{V\fy,\fy}$, so that the linear equation can make sense in $H^{-1}(\R^d)$ for $H^1(\R^d)$ initial data (which is extended to $L^2(\R^d)$ initial data by density).

In Theorem~\ref{T2}, the assumptions on the solution $u$ imply that $u$ is uniformly bounded in $L^2$, and hence by the discussion above we get
\begin{equation} \label{unifST}
 \sup_{t_0>0} \|u\|_{ST(t_0,t_0+T)} \leq C \|u\|_{L^\infty_tL^2_x(0,\infty)} < \infty,
\end{equation}
for some $T>0$ that is uniform with respect to $\|u\|_{L^\infty_tL^2_x}$.

\section{Proof of Theorem~\ref{T2}}\label{S:T2}

In this section we will prove Theorem~\ref{T2}.  As mentioned above, the heart of the argument is the same as in \cite{Glassey, GlasseyKG, Strauss, Barab}:  assuming that $u$ admits the decomposition \eqref{u-decomposition} with nonzero $v_+$, we will show that for a suitable choice of $w(t)=e^{it\Delta}\varphi$, the derivative of the quantity $\langle u(t),w(t)\rangle$ (which is uniformly bounded in time) may be written as a main term, which is of size $t^{-\frac{dp}{2}}$, plus errors that are small compared to this main contribution.  As $t^{-\frac{dp}{2}}$ is not integrable in time for $p\leq \tfrac{2}{d}$, this yields a contradiction.

\begin{proof}[Proof of Theorem~\ref{T2}] 
We assume towards a contradiction that $u$ has the decomposition appearing in \eqref{u-decomposition} with some nonzero $v_+\in L^2$. Without loss of generality, we may assume $q<2$ is as close to $2$ as we will need.  Recall that $u\in L^\infty((0,\infty);L^2(\R^d))$ due to the assumed asymptotic behavior.  

Let $w=e^{it\Delta}\fy$ for some $\fy\in \S(\R^d)$ to be chosen below.  In particular, by the dispersive estimate, 
\begin{equation}\label{dispersive-for-w}
\|w(t)\|_{L^r(\R^d)} \leq C t^{\frac{d}{r}-\frac{d}{2}}\qtq{for} 2\leq r\leq\infty.
\end{equation}

Define
\begin{equation}\label{tildes}
 u=MD\ti u, \quad w=MD\ti w, \quad l=MD\ti l,
 \end{equation}
where $M=M(t)$ and $D=D(t)$ are as in Section~\ref{S:preliminaries}. Then 
\[
 \|u-w\|_2 = \|\ti u-\ti w\|_2 \qtq{and} \|\ti l\|_q = t^{-(\frac{d}{q}-\frac{d}{2})}\|l\|_q \to 0. \]

We begin by using \eqref{nls} and a change of variables to compute
\begin{equation}\label{time-derivative}
\begin{aligned}
 i\pt_t\LR{u,w} &= \LR{F(u),w} + \LR{Vu,w}
 \\ &= t^{-\frac{dp}{2}}\LR{F(\ti u),\ti w}+ \LR{Vu,w},
\end{aligned}
\end{equation}
where 
\[
F(u) = \mu |u|^p u \qtq{or} F(u) = \mu (|x|^{-\frac{dp}{2}}\ast |u|^2)u. 
\]

\subsection{Analysis of the potential term} We first treat the potential term.  Using \eqref{cond V}, we may decompose $V=V_1+V_2$ so that
\[
 V_1\in L^\infty_t L^{\frac{2}{p}-}_x\qtq{and} V_2\in\begin{cases} L^\infty_t L^{\frac{d}{2}}_x &(d\ge 3),\\
 L^\infty_t L^{1+}_x &(d=2),\\
 L^\infty_t \M_x &(d=1). 
 \end{cases} \] 
 
The contribution of $V_1$ is controlled as follows:  by H\"older's inequality and \eqref{dispersive-for-w}, we have
\begin{equation} \label{V2 bd}
 |\LR{V_1u,w}| \le \|V_1\|_{L^\infty_tL^{\frac{2}{p}-}}\|u\|_{L^\infty_tL^2}\|w(t)\|_{L^{\frac{2}{1-p}+}}
 \leq C \|u\|_{L^\infty_tL^2}\cdot t^{-\frac{dp}{2}-},
\end{equation}
where we have used the condition $p<1$. 
  
For the $V_2$ part, we will prove an estimate for the integral over $[0,\tau]$ for arbitrary $\tau\geq 2$ (rather than a pointwise-in-time estimate); as we will ultimately show that the main term in \eqref{time-derivative} is of size $t^{-\frac{dp}{2}}$, we seek to establish a bound for the integral that is $o(\alpha(\tau))$, where
\begin{equation}\label{alpha}
\alpha(\tau) = \begin{cases} \tau^{1-\frac{dp}{2}} & p<\tfrac{2}{d} \\ \log\tau & p=\tfrac{2}{d}.\end{cases}
\end{equation}  

We recall the local-existence time $T$ from Section~\ref{S:preliminaries} and the local Strichartz bound \eqref{unifST}. 

First, 
if $d=1$ then for any $t_0>0$ we have 
\[
 \int_{t_0}^{t_0+T} |\LR{V_2u,w}| dt \leq T^{\frac34} \|V_2\|_{L^\infty_t \M} \|u\|_{L^4_tC_b} \|w\|_{L^\infty_tC_b(t_0,\infty)}  = O(t_0^{-\frac12}).\]
Similarly, for $d=2$ we have
\[
 \int_{t_0}^{t_0+T} |\LR{V_2u,w}| dt \leq T^{\frac12+} \|V_2\|_{L^\infty_t L^{1+}} \|u\|_{L^{2+}_tL^{\infty-}} \|w\|_{L^\infty_{t,x}(t_0,\infty)} = O(t_0^{-1}),
 \]
while if $d=3$ we have 
\[
 \int_{t_0}^{t_0+T} |\LR{V_2u,w}| dt \leq \|V_2\|_{L^\infty_tL^{\frac{d}{2}}}T^{\frac34}\|u\|_{L^4_t L^3(t_0,t_0+T)}\|w\|_{L^\infty_{t,x}(t_0,\infty)}
 =O(t_0^{-3/2}).
 \]
For $d\ge 4$, we can actually obtain an integrable pointwise-in-time estimate, namely
\[
 |\LR{V_2u,w}| \leq \|V_2\|_{L^\infty_tL^{\frac{d}{2}}}\|u\|_{L^\infty_tL^2}\|w(t)\|_{L^{\frac{2d}{d-4}}} 
 =O(t^{-2}). \]
Thus we obtain for $\tau>2$, 
\begin{equation} \label{V1 bd}
 \int_0^{\tau}|\LR{V_2u,w}|dt = \begin{cases} O(\tau^{\frac12}) &(d=1),\\ O(\log \tau) &(d=2),\\ O(1) &(d\ge 3),\end{cases}
\end{equation}
which in all cases is  $o(\alpha(\tau))$ provided $p\leq \tfrac{2}{d}$ and $p<1$.  This completes the treatment of the potential term.

We next turn to the main term $\langle F(\tilde u),\tilde w\rangle$. 

\subsection{Analysis of the power-type nonlinearity.} For the power-type nonlinearity, we claim that
\begin{equation}\label{Main-term1}
\langle F(\tilde u),\tilde w\rangle = \langle F(\hat v_+),\hat\varphi\rangle + o(1) \qtq{as}t\to\infty, 
\end{equation}
for which it suffices to prove
\begin{align}
\tilde w &  =\hat \fy + o(L^2\cap L^\infty), \label{J1} \\
F(\tilde u) & = F(\hat v_+) +o(L^{\frac{2}{1+p}}+L^{\frac{2}{1+p}-})\label{J2}
\end{align}
as $t\to\infty$. Here we are using the notation $o(X)$ from Section~\ref{S:preliminaries}. Note that to choose an exponent $\tfrac{2}{1+p}-\in[1,\infty]$ requires $p<1$. 

To prove \eqref{J1}, observe that by definition of $\tilde w$ and \eqref{factorization} we have $\tilde w= \F M\varphi$; thus we may use Hausdorff--Young to estimate
\begin{equation} \label{w 2 fy}
 \|\ti w-\hat\fy\|_r \le C\|(M-1)\fy\|_{r'} \le Ct^{-1}\||x|^2\fy\|_{r'}
\end{equation}
uniformly in $t$ for any $r\in[2,\infty]$.  

For \eqref{J2}, we firstly observe that by \eqref{factorization} we have
\begin{align*}
u-e^{it\Delta}v_+ - l & = MD[\tilde u -\hat v_+-\tilde l]+MD\F[M-1] v_+,
\end{align*}
so that by hypothesis and dominated convergence we obtain
\begin{equation}\label{u-decomposition2}
\tilde u - \hat v_+-\tilde l = o(L^2)\qtq{as}t\to\infty.
\end{equation}
As $\tilde l$ is $o(L^{2-})$, we can therefore use H\"older to estimate
\begin{multline*}
 \|F(\ti u)-F(\hat v_+)\|_{L^{\frac{2}{1+p}}+L^{\frac{2}{1+p}-}}
 \le C\|(|\ti u-\ti l-\hat v_+|+|\ti l|)(|\ti u|+|\hat v_+|)^p\|_{L^{\frac{2}{1+p}}+L^{\frac{2}{1+p}-}}
 \\ \le C[\|\ti u-\ti l-\hat v_+\|_2+\|\ti l\|_{2-}][\|\ti u\|_2+\|\hat v_+\|_2]^p \to 0
\end{multline*}
as $t\to\infty$.

As $\hat v_+\neq 0$, we must have $F(\hat v_+)\not=0$, so that we may now choose $\fy\in \S(\R^d)$ satisfying 
\[
\LR{F(\hat v_+),\hat\fy}>0.
\] 
Then, continuing from \eqref{time-derivative} and using \eqref{Main-term1}, \eqref{V2 bd}, and \eqref{V1 bd}, we have
\begin{align*}
\int_1^\tau \Re i \partial_t\LR{u,w}\,dt = \int_1^\tau t^{-\frac{dp}{2}}(\LR{F(\hat v_+),\hat\fy}+o(1))\,dt+ o(\alpha(\tau)) \geq C\alpha(\tau) 
\end{align*}
as $\tau\to\infty$, where $\alpha(\cdot)$ is as in \eqref{alpha}. As $\alpha(t)\to\infty$ as $t\to\infty$, this contradicts the fact that $|\LR{u,w}| \le \|u\|_2\|\fy\|_2$ is bounded uniformly in time. 

\subsection{Analysis of the Hartree-type nonlinearity} It remains to treat the case of the Hartree nonlinearity in \eqref{time-derivative}.  This case is a bit more subtle than the power-type case, as we must incorporate the contribution of $\tilde l$ to identify the asymptotic behavior of $\langle F(\tilde u),\tilde w\rangle$.  For example, if $l$ has the form $|l(t,x)|=Q(x-ct)$, then 
\[
 |\ti l(t,x)|^2 = t^d|l(t,tx)|^2 = t^d Q^2(t(x-c)) \to \|Q\|_2^2 \delta(x-c)
 \]
as $t\to\infty$, which gives a non-trivial contribution to $\LR{F(\ti u),\ti w}$.  

Writing 
\[
\langle F(\tilde u),\tilde w\rangle = \langle |x|^{-\frac{dp}{2}}\ast |\tilde u|^2,\bar{\tilde u} \tilde w \rangle,
\]
we will show
\begin{align}
\overline{\tilde u}\tilde\omega & = \overline{\hat v_+}\hat\varphi + o(L^1\cap L^{2-}),\label{J4} \\
|x|^{-\frac{dp}{2}}\ast |\tilde u|^2 & = |x|^{-\frac{dp}{2}}\ast(|\hat v_+|^2 + |\tilde l|^2)+o(L^{\frac{2}{p},\infty}),\label{J3}
\end{align} 
where we once again recall the notation $o(X)$ from Section~\ref{S:preliminaries}. 

To prove \eqref{J4} we use \eqref{w 2 fy}, \eqref{u-decomposition2}, and H\"older to write
\begin{align*}
 \overline{\ti u}\ti w &= [\overline{\hat v_+}+ o(L^{2-}) + o(L^2)][\hat \fy+ o(L^2\cap L^\infty)] 
 \\& = \overline{\hat v_+} \hat \fy + o(L^1\cap L^{2-}).
\end{align*}

For \eqref{J3} we observe that as $\|\ti l\|_2$ is bounded and $\|\ti l\|_q\to 0$ as $t\to\infty$, we have $|\ti l|\to 0$ weakly in $L^2$.  Thus
\[
 \|\hat v_+\ti l\|_1 = \LR{|\hat v_+|,|\ti l|} \to 0 \qtq{as}t\to\infty,
 \]
which yields
\[
 |\ti u|^2 = |\hat v_+ + \ti l + o(L^2)|^2 = |\hat v_+|^2 + |\ti l|^2 + o(L^1).
 \]
Using the triangle inequality for the $L^{\frac{2}{p},\infty}$ norm, translation invariance of the weak Lebesgue space, and the fact that $|x|^{-\frac{dp}{2}}\in L^{\frac{2}{p},\infty}$, this implies
\[
 |x|^{-\frac{dp}{2}}*|\ti u|^2 = |x|^{-\frac{dp}{2}}*(|\hat v_+|^2+|\ti l|^2) + o(L^{\frac{2}{p},\infty}),
 \]
 yielding \eqref{J3}. 
 
 Now, since $p<1$, we have $(\tfrac{2}{p})'=\tfrac{2}{2-p}\in(1,2)$, so that $L^1 \cap L^{2-} \subset L^{\frac{2}{2-p},1}$.  Thus, \eqref{J4} and \eqref{J3} together imply
 \[
 \tfrac{1}{\mu}\langle F(\tilde u),\tilde w\rangle = \langle |x|^{-\frac{dp}{2}}\ast(|\hat v_+|^2+|\tilde l|^2),\overline{\hat v_+}\hat\varphi\rangle + o(1)
 \]
 as $t\to\infty$.  We now define 
\[
  \hat v_+'(x) := \begin{cases} \hat v_+(x) &|\hat v_+(x)|\le n\\ 0 &|\hat v_+(x)|>n, \end{cases} 
  \]
where (recalling $\hat v_+\neq 0$)  we choose $n$ large enough so that $\hat v_+'$ is non-zero in $L^2\cap L^\infty\subset L^{\frac{2}{1-p},2}$. We write
\begin{align*}
\tfrac{1}{\mu}\langle F(\tilde u),\tilde w\rangle & = \langle |x|^{-\frac{dp}{2}}\ast(|\hat v_+|^2 + |\tilde l|^2),|\hat v_+'|^2\rangle \\
&\quad +\langle |x|^{-\frac{dp}{2}}\ast(|\hat v_+|^2+|\tilde l|^2),\overline{\hat v_+}[\hat\varphi-\hat v_+']\rangle \\
& \quad + \langle |x|^{-\frac{dp}{2}}\ast(|\hat v_+|^2+|\tilde l|^2), [\overline{\hat v_+ - \hat v_+'}]\hat v_+'\rangle + o(1)
\end{align*}
and observe that the third term on the right-hand side vanishes identically by construction. Noting that $\F \S$ is dense in $L^{\frac{2}{1-p},2}$, we now choose $\varphi\in\S$ so that
\[
\bigl|\langle |x|^{-\frac{dp}{2}}\ast (|\hat v_+|^2+|\tilde l|^2),\overline{\hat v_+}[\hat\varphi-\hat v_+']\rangle\bigr| \leq C \|u\|_{L_t^\infty L^2}^2 \|v_+\|_2 \|\hat\varphi-\hat v_+'\|_{L^{\frac{2}{1-p},2}}
\]
is as small as we wish. Then we finally get
\[
\Re \tfrac{1}{\mu}\langle F(\tilde u),\tilde w\rangle \geq \tfrac12\langle |x|^{-\frac{dp}{2}}\ast |\hat v_+|^2,|\hat v_+'|^2\rangle + o(1)
\]
as $t\to\infty$.  This leads to a contradiction just as in the power-type case and hence completes the proof of Theorem~\ref{T2}.\end{proof}

\section{Proof of Theorem~\ref{T3}}\label{S:Extension}

The proof of Theorem~\ref{T2} breaks down at $p=1$ both in treating the potential term and the localized part $l$.  The Hartree case has another issue, namely, that $F(u)$ cannot be controlled by $u\in L^2$ in the distributional sense. Thus in this section we will consider only the power-type nonlinearity with $p=1\leq \tfrac{2}{d}$. 
\begin{proof}[Proof of Theorem~\ref{T3}] Suppose towards a contradiction that $u$, $V$, $l$, $v_+$, and $\nu$ are as in Theorem~\ref{T3}, but $v_+\neq 0$. 

The general strategy is the same as in the proof of Theorem~\ref{T2}.  As in that proof, we let $w(t) = e^{it\Delta}\varphi$ for $\varphi\in \S(\R^d)$ to be determined below and introduce $\tilde u$, $\tilde w$, and $\tilde l$ as in \eqref{tildes}.  Let us also denote
\[
\tilde V(t,x) = t^{\frac{d}{2}} V(t,tx). 
\]

 As in \eqref{time-derivative}, we have
\[
i\partial_t \langle u, w\rangle = t^{-\frac{d}{2}}\langle F(\tilde u),\tilde w\rangle+\langle Vu,w\rangle.
\]
Recalling \eqref{J1} and observing that
\[
\tfrac{1}{\mu} F(\tilde u) = |\tilde u| u = |\hat v_+|\hat v_+ + |\tilde l| \tilde l+o(L^1)
\]
(cf. \eqref{u-decomposition2}), we can write
\begin{equation}\label{Zmain1}
\tfrac{1}{\mu}\langle F(\tilde u), \tilde w\rangle = \langle |\hat v_+| \hat v_+,\hat\varphi\rangle + \langle |\tilde l|\tilde l,\hat\varphi\rangle+o(1). 
\end{equation}

Below we will construct $\varphi\in \mathcal{S}(\R^d)$ such that the `main term' coming from $\hat v_+$ satisfies
\begin{equation}\label{Zlb}
\|\hat v_+\|_{L^2}^2 \sim \langle |\hat v_+|\hat v_+,\hat\varphi\rangle \gg \langle \nu,|\hat\varphi|\rangle.
\end{equation}
On the other hand, the assumption \eqref{T3-hypothesis} guarantees 
\[
| \langle |\tilde l| \tilde l,\hat\varphi\rangle| \leq \langle |\tilde l|^2,|\hat\varphi|\rangle \leq \langle \nu, |\hat \varphi|\rangle + o(1),
\]
which is therefore dominated by the main contribution coming from $\hat v_+$. 

We will control the potential term by 
\begin{equation}\label{control-potential}
t^{-\frac{d}{2}}\langle \nu,|\hat\varphi|\rangle+o(t^{-\frac{d}{2}}),
\end{equation}
which is also dominated by the main term.  In particular, we will have that $\Re i\partial_t\langle u,w\rangle$ has a non-integrable lower bound of $Ct^{-\frac{d}{2}}$ (cf. $d\in\{1,2\}$), thus leading to the same contradiction as in the proof of Theorem~\ref{T2}. 

For the estimate of the potential term, we first write
\begin{align*}
|\langle Vu,w\rangle| & \leq |\langle V(u-l),w\rangle| + t^{-\frac{d}{2}}|\langle \tilde V \tilde l,\tilde w\rangle|.
\end{align*}
Now, given $\eps>0$ we may select $\tilde v_+ \in \mathcal{S}(\R^d)$ so that 
\[
\|v_+-\tilde v_+\|_{L^2}<\eps. 
\]
We can then use \eqref{u-decomposition} and the dispersive estimate to write
\begin{align*}
|\langle V(u-l),w\rangle| & \leq |\langle V e^{it\Delta}\tilde v_+,w\rangle| + |\langle V e^{it\Delta}[v_+-\tilde v_+],w\rangle| + |\langle V\cdot o(L^2),w\rangle| \\
& \leq \|V\|_{L_t^\infty L^{2-}}\|e^{it\Delta} \tilde v_+\|_{L^{2+}}\|w\|_{L^\infty} + \|V\|_{L_t^\infty L^2}\|w\|_{L^\infty}[\eps+o(1)] \\
& = o(t^{-\frac{d}{2}})+ C\eps t^{-\frac{d}{2}},
\end{align*}
which is an acceptable contribution to \eqref{control-potential}.  On the other hand, recalling \eqref{J1} and applying \eqref{T3-hypothesis},  we get
\begin{align*}
t^{-\frac{d}{2}}|\langle \tilde V \tilde l,\tilde w\rangle|& \leq \tfrac12 t^{-\frac{d}{2}}\{ \langle |\tilde V|^2,|\hat\varphi|\rangle + \langle|\tilde l|^2,|\hat \varphi|\rangle\} + o(t^{-\frac{d}{2}})\cdot\| V\|_{L_t^\infty L^2} \| l\|_{L_t^\infty L^2} \\
& \leq t^{-\frac{d}{2}}\langle \nu,|\hat\varphi|\rangle + o(t^{-\frac{d}{2}}),
\end{align*}
which is again an acceptable contribution to \eqref{control-potential}.

To complete the proof of Theorem~\ref{T3}, it therefore remains to find $\varphi\in\mathcal{S}(\R^d)$ so that \eqref{Zlb} holds.  We turn to this now.
%
%

As $\nu$ is singular with respect to Lebesgue measure, there exists a null set $N\subset\R^d$ such that $\nu(N^c)=0$. Then fixing $\eps>0$, there is a sequence of balls $\{B_k\}\subset\R^d$ such that $N\subset U:=\bigcup_k B_k$ and $\sum_k|B_k|<\eps.$  

As $\supp\nu\subset U$ and $\nu$ is a finite measure, there exists $n\in\N$ such that $\nu(\bigcup_{k>n}B_k)<\eps$. Thus, defining  $W:=\bigcup_{k\le n}B_k,$ we have $\nu(W^c)<\eps$.  

We now claim that there exists a cut-off function $\psi\in C_c^\infty(\R^d)$ such that $0\le\psi\le 1$, $\psi= 1$ on $W$, and $\|\psi\|_{L^1}\lesssim\eps$. To see this, take $\chi\in C_c^\infty(|x|\le 2)$ satisfying $\chi(x)=1$ for $|x|\le 1$ and take $\lambda\in C^\infty$ satisfying $\lambda(t)=1$ for $t\ge 1$ and $\lambda(t)=0$ for $t\le 1/2$.  Writing $B_k=B(c_k,r_k)$, we can then set
\[
 \psi(x)=\lambda\bigl(\sum_{k\le n}\chi(\tfrac{x-c_k}{r_k})\bigr).
 \]

Choosing a sequence of $\eps_n\to 0$, we can therefore obtain a sequence $\psi_n\in C_c^\infty(\R^d)$ such that $0\le\psi_n\le 1$, $\supp\psi_n$ is contained in a fixed bounded set, $\|\psi_n\|_{L^1}\to 0$, and $\psi_n=1$ on some open set $W_n$ satisfying $\nu(W_n^c)\to 0$. Passing to a subsequence, we may also assume $\psi_n\to 0$ a.e. on $\R^d$. 

Now take a sequence $\phi_n\in \S(\R^d)$ such that $\hat\phi_n$ is bounded in $L^\infty(\R^d)$ and 
\[
\LR{|\hat v_+|\hat v_+,\hat\phi_n}\to\||\hat v_+|\hat v_+\|_1=\|\hat v_+\|_2^2.
\]
We then define $\varphi_n\in\S(\R^d)$ by $\hat\varphi_n=(1-\psi_n)\hat\phi_n$. Using $\hat v_+\in L^2$, $\hat\phi_n$ is bounded in $L^\infty$, and $\psi_n\to 0$ a.e., the dominated convergence theorem implies 
\[
 \LR{|\hat v_+|\hat v_+,\hat\varphi_n} \to \|\hat v_+\|_2^2 >0, 
 \]
while 
\[
 \LR{\nu,|\hat\varphi_n|} \le  \nu(W_n^c) \|\hat\fy_n\|_\infty\to 0.
 \]
Hence for large $n$, we have 
\[
 \|\hat v_+\|_2^2 \sim \LR{|\hat v_+|\hat v_+,\hat\varphi_n} \gg  \LR{\nu,|\hat\varphi_n|}, 
 \]
and so we can achieve \eqref{Zlb} by choosing $\varphi=\varphi_n$ for sufficiently large $n$.  This completes the proof of Theorem~\ref{T3}.\end{proof}

\end{document}